\documentclass{elsarticle}
\usepackage{natbib}
\usepackage[english]{babel}
\usepackage{ amsfonts,amsthm,amsmath,amssymb,dsfont}
\usepackage[T1]{fontenc}
\usepackage[cp1250]{inputenc}
\usepackage{graphicx}

\usepackage[all]{xy}

\newtheorem{thm}{Theorem}
\newtheorem{lem}[thm]{Lemma}
\newtheorem{Definition}[thm]{Definition}
\newtheorem{proposition}[thm]{Proposition}
\newtheorem{Corollary}[thm]{Corollary}

\newdefinition{rmk}{Remark}
\newproof{pf}{Proof}
\newproof{pot}{Proof of Theorem \ref{thm2}}

\newcommand{\R}{\mathbb{R}}
\newcommand{\N}{\mathbb{N}}

\DeclareMathOperator{\supp}{supp}
\DeclareMathOperator{\rank}{rank}
\DeclareMathOperator{\id}{id}
\DeclareMathOperator{\Id}{Id}

\begin{document}

\cortext[cor1]{Corresponding author}

\title{$s$-Numbers of compact embeddings of function spaces on quasi-bounded domains}
\author{Shun Zhang}\author{Alicja G\k{a}siorowska\corref{cor1}}

\begin{abstract}
We prove asymptotic formulas for the behavior of approximation, Gelfand, Kolmogorov and Weyl numbers of Sobolev embeddings between Besov and Triebel-Lizorkin spaces defined on quasi-bounded domains.
\end{abstract}

\maketitle

\section{Introduction}
Today we have a good knowledge about asymptotic behavior of entropy and approximation numbers of compact embeddings between spaces of Sobolev-Besov-Hardy type. The results regarding  asymptotic behavior of the entropy numbers and the approximation numbers of the embeddings of spaces defined on bounded domains are the oldest part of this theory. We owe  M.S.~Birman,   M.Z.~Solomjak \cite{BS}, D.~Edmunds, H.~Triebel \cite{ET}, A.~Caetano \cite{Cae} and others  the estimates in this respect. The Gelfand, the Kolmogorov and the Weyl numbers on bounded domains were studied by  C.~Lubitz \cite{Lub}, R.~Linde \cite{RL}, A.~Caetano \cite{Cae_w1,Cae_w2} and J.~Vyb\'iral \cite{JV}. But, there is a wider class of domains for which the embeddings of spaces of Besov-Sobolev type can be compact. These are so called quasi-bounded domains. The quasi-bounded domain may  be unbounded, and can be even a domain of infinite Lebesgue measure. Recently H.-G.~Leopold and L.~Skrzypczak \cite{LS} gave necessary and sufficient conditions for compactness of these embeddings and studied the asymptotic behavior of entropy numbers for function spaces defined on quasi-bounded domains. In this article we investigate asymptotic behavior of  $s$-numbers of these embeddings. Our approach is based on wavelet decomposition of function spaces on domains introduced by H.~Triebel in 2008; cf. \cite{HT}. In the case of the Weyl numbers our estimates complement the earlier results even for bounded domains.

The concept of quasi-bounded domains is not new. Function spaces on these domains were studied in the 60s of the last century by  Clark and Hewgill; see \cite{Cla,CH}. Overview of some of these results can be found in  \cite{AF}. All these results assume a certain regularity of the boundary of the domain. The assumptions that we make here on the regularity of the boundary (uniformly $ E $-porous domain) are weaker than the previously assumed.

\section{Function spaces on quasi-bounded domain}
We assume that the reader is acquainted with the definition and basic properties of Besov spaces $B_{p,q}^s(\R^d)$, $0<	p,q\leq \infty$, $s\in\R$ and Triebel-Lizorkin spaces $F_{p,q}^s(\R^d)$, $0< p<\infty$, $0<q\leq\infty$, $s\in\R$.
We denote the continuous embeddings between quasi-Banach spaces by $\hookrightarrow$ and also we  use the common notations $A_{p,q}^s(\R^d)$, $A_{p,q}^s(\Omega)$ with $A=B$ or $A=F$, if it makes no difference.

Let $\Omega$ be an open  and nonempty set in $\R^d$. Such set is called a domain.
 The domain $\Omega$ in $\R^d$ is called quasi-bounded if
 \begin{equation}
\lim_{x\in\Omega, |x|\rightarrow\infty} d(x,\partial\Omega)=0.
\label{quasi}
\end{equation}
 In particular, each bounded domain is quasi-bounded.

\begin{Definition}
Let $\Omega$ be the domain in $\R^d$, such that $\Omega\neq\R^d$ and let $0<p,q\leq \infty$, $s\in\R$ with  $p<\infty$ for $F$-spaces.
\begin{enumerate}[(i)]
\item Let
$$A_{p,q}^s(\Omega)=\{f\in D'(\Omega): f=g|_\Omega\ \text{for some distribution}\ g\in A_{p,q}^s(\R^d)\},$$
$$||f| A_{p,q}^s(\Omega)||=\inf||g|A_{p,q}^s(\R^d)||,$$
where the infimum is taken over all $g\in A_{p,q}^s(\R^d)$ with $f=g|_\Omega$.
	\item Let
	$$\widetilde{A}_{p,q}^s(\bar{\Omega})=\{f\in A_{p,q}^s(\R^d): \supp f\subset \bar{\Omega}\},$$
	$$\widetilde{A}_{p,q}^s(\Omega)=\{f\in D'(\Omega): f=g|_\Omega\ \text{for some distribution}\ g\in \widetilde{A}_{p,q}^s(\bar{\Omega})\},$$
	$$\|f|\widetilde{A}_{p,q}^s(\Omega)\|=\inf\|g|A_{p,q}^s(\R^d)\|,$$
	where the infimum is taken over all $g\in\widetilde{A}_{p,q}^s(\bar{\Omega})$ with $f=g|_\Omega $.
	\item We define
	$$
	\bar{B}_{p,q}^s(\Omega)=\begin{cases}
\widetilde{B}_{p,q}^s(\Omega), &\text{if} \quad 0<p\leq\infty, 0< q\leq\infty, s>\sigma_{p},\\
B_{p,q}^0(\Omega), &\text{if}\quad 1<p<\infty, 1\leq q\leq\infty, s=0,\\
 B_{p,q}^s(\Omega), &\text{if} \quad 0< p<\infty, 0< q\leq\infty, s<0,
\end{cases}
$$
and
$$
\bar{F}_{p,q}^s(\Omega)=\begin{cases}
\widetilde{F}_{p,q}^s(\Omega), &\text{if} \quad 0< p<\infty, 0< q\leq\infty, s>\sigma_{p,q},\\
F_{p,q}^0(\Omega), &\text{if}\quad 1<p<\infty, 1\leq q\leq\infty, s=0,\\
 F_{p,q}^s(\Omega), &\text{if} \quad 0< p<\infty, 0< q\leq\infty, s<0,
\end{cases}	
$$
where $$\sigma_p=d\bigg(\frac{1}{p}-1\bigg)_+,\quad \sigma_{p,q}=d\bigg(\frac{1}{\min(p,q)}-1\bigg)_+, \ 0< p,q\leq\infty.$$
\end{enumerate}
\label{aomega}
\end{Definition}
In this paper we consider only the so called  uniformly $E$-porous domain introduced by H. Triebel; cf. \cite{HT}. This assumption of the uniformly $E$-porous property allows us to use the wavelet characterization of these function spaces.
\begin{Definition}
\begin{enumerate}[(i)]
	\item A close set $\Gamma\subset\R^d$ is said to be porous if there exists a number  $0<\eta<1$ such that one finds for any ball $B(x,r)\subset\R^d$ centered at $x\in\Gamma$ and of radius $r$ with $0<r<1$, a ball $B(y,\eta r)$ with
$$B(y,\eta r)\subset B(x,r)\quad \text{and}\quad B(y,\eta r)\cap\Gamma=\emptyset.$$
\item
A close set $\Gamma\subset\R^d$ is said to be uniformly porous if it is porous and there is a locally finite positive Radon measure $\mu$ on $\R^d$ such that $\Gamma =\supp \mu$ and for some constants $C,c>0$
$$C\ h(r)\leq\mu(B(x,r))\leq c\ h(r), \quad \text{for}\quad x\in\Gamma, 0<r<1,$$
where $h:[0,1]\rightarrow\R$ is a continuous strictly increasing function with  $h(0)=0$ and $h(1)=1$ (the constants $C,c$ are independent of $x$ and $r$).
\end{enumerate}
\label{por}
\end{Definition}

\begin{Definition}
Let $\Omega$ be an open set in $\R^d$ such that $\Omega\neq\R^d$ and $\Gamma=\partial\Omega$.
\begin{enumerate}[(i)]
	\item The domain $\Omega$ is said to be  $E$-porous if there is a number $0<\eta<1$ such that one finds for any ball $B(x,r)\subset\R^d$ centered at $x\in\Gamma$ and of radius $r$ with $0<r<1$, a ball  $B(y,\eta r)$ with
$$B(y,\eta r)\subset B(x,r)\quad \text{and}\quad B(y,\eta r)\cap\bar{\Omega}=\emptyset.$$
\item The domain $\Omega$ is called uniformly $E$-porous  if it is $E$-porous and $\Gamma=\partial\Omega$ is uniformly porous.
\end{enumerate}
\label{Epor}
\end{Definition}

In 2008 H. Triebel presented wavelet characterization of function spaces on E-porous domains; cf. \cite[Chapter 2 and 3]{HT}. By Theorem 3.23 in \cite{HT} there exists an isomorphism between spaces $\bar{B}_{p,q}^s(\Omega)$ and $\ell_q\big(2^{j(s-{d}/{p})}\ell_p^{M_j}\big)$, where the last space is defined below.

\begin{Definition}
Let $0<p\leq \infty, 0<q\leq\infty$, $\{\beta_j\}_{j=0}^\infty$ be a sequence of positive numbers and $\{M_j\}_{j=0}^\infty\in\bar{\N}$, $j\in\N_0$, $\bar{\N}=\N\cup\{\infty\}$.  Then
$$
 \ell_q(\beta_j\ell_p^{M_j})=\bigg\{x\ :\ x=\{x_{j,l}\}_{j\in\N_0, l=1,\ldots,M_j }\quad \mathrm{with}$$
$$||x|\ell_q(\beta_j\ell_p^{M_j})||=\bigg(\sum_{j=0}^\infty \beta_j^q\bigg(\sum_{l=1}^{M_j}|x_{j,l}|^p\bigg)^{q/p}\bigg)^{1/q}<\infty\bigg\}
$$
with the obvious modifications if $p=\infty$ or $q=\infty$.
\label{ciagi2}
\end{Definition}

In \cite{LS} H.-G. Leopold and L. Skrzypczak  introduced also the so called box packing constant of a domain in $\R^d$.

Let $\Omega\subset\R^d$ be the domain such that  $\Omega\neq\R^d$. We put
\begin{multline*}
b_j(\Omega)=\sup\bigg\{k:\bigcup_{\ell=1}^kQ_{j,m_\ell}\subset\Omega, \\
 Q_{j,m_\ell}\ \text{being pairwise disjoint dyadic cubes of size}\ 2^{-j}\bigg\},
 \end{multline*}
for $j=0,1,\ldots$ and we put
$$b(\Omega)=\sup\bigg\{t\in\R_+:\limsup_{j\rightarrow\infty}b_j(\Omega)2^{-jt}=\infty\bigg\}.$$
 Moreover it is proved in \cite{LS}, that for any nonempty open set $\Omega\subset\R^d$ we have $d\leq b(\Omega)\leq\infty$. If $\Omega$ is unbounded and is not quasi-bounded, then $b(\Omega)=\infty$. If  Lebesgue measure $|\Omega|$ is finite, then $b(\Omega)=d$. In particular if the domain is bounded, then $b(\Omega)=d$. If the domain is quasi-bounded and its  Lebesgue measure is infinite then $b(\Omega)$ can be finite and bigger than $d$. However there are quasi-bounded domains such that $b(\Omega)=\infty$.

Furthermore in \cite{LS} it is shown that
\begin{equation}
M_j\sim b_j(\Omega).
\label{l_kostek}
\end{equation}

 The  box packing constant helps to prove the conditions for the compactness of Sobolev embeddings.
 \begin{thm}
 Let $\Omega$ be an uniformly $E$-porous quasi-bounded domain in $\R^d$ and let $\frac{1}{p^*}=\bigg(\frac{1}{p_2}-\frac{1}{p_1}\bigg)_+$, $b(\Omega)<\infty$. Then
 \begin{equation}
 \label{eq:1}
 \bar{B}_{p_1,q_1}^{s_1}(\Omega) \hookrightarrow\bar{B}_{p_2,q_2}^{s_2}(\Omega)
 \end{equation}
 is compact if
 $$s_1-s_2-d\bigg(\frac{1}{p_1}-\frac{1}{p_2}\bigg)>\frac{b(\Omega)}{p^*}.$$
 If the embedding (\ref{eq:1}) is compact and $\frac{1}{p^*}=0$ then $s_1-s_2-d\bigg(\frac{1}{p_1}-\frac{1}{p_2}\bigg)>0$.\\
 If the embedding (\ref{eq:1}) is compact and $\frac{1}{p^*}>0$ then $s_1-s_2-d\bigg(\frac{1}{p_1}-\frac{1}{p_2}\bigg)\geq \frac{b(\Omega)}{p^*}$.
 \label{twr:1.4}
 \end{thm}

Now we will formulate the key definitions, where we denote by $L(X,Y)$ the class of all linear continuous operators from $X$ into $Y$.

\begin{Definition}
Let $X$ and $Y$ be quasi-Banach spaces and $T~\in~L(X,Y)$.
\begin{enumerate}[(i)]
	\item For $k\in \N$, we define the \textit{$k$-th approximation number} by
	\begin{displaymath}
	a_k(T):=\inf\{\|T-A\|\ :\ A\in L(X,Y),\ \rank(A)<k\},
	\end{displaymath}
	where $\rank(A)$ denotes the dimension of the range $A(X)=\{A(x), x\in X\}$.
	\item For $k\in \N$, we define the \textit{$k$-th Gelfand number} by
	\begin{displaymath}
	c_k(T):=\inf\{\|TJ_M^X\|: M\subset X,\ \mathrm{codim}(M)<k\},
	\end{displaymath}
	where $J_M^X$ is the natural injection of $M$ into  $X$ and  $M$ is a closed subspace of the quasi-Banach space $X$.
	\item For $k\in \N$, we define the \textit{$k$-th Kolmogorov number} by
	\begin{displaymath}
	d_k(T):=\inf\{\|Q_N^YT\|: N\subset Y,\ \mathrm{dim}(N)<k\},
	\end{displaymath}
	where $Q_N^Y$ is the natural surjection of $Y$ onto  the quotient space
	
$$Y/N=\{y+N:\|y+N\|=\inf_{z\in N}\|y+z\|\}$$
and $N$  is a closed subspace of the quasi-Banach space $Y$.
\item For $k\in \N$, we define the \textit{$k$-th Weyl number} by
	$$
	x_k(T):=\sup\{a_k(TS): S\in L(\ell_2, X)\ \text{with}\ \|S\|\leq 1\}.
	$$
\end{enumerate}
\end{Definition}

The approximation numbers, the Gelfand, the Kolmogorov and the Weyl numbers  in the context of Banach spaces are examples of $s$-numbers, that were introduced by A. Pietsch; cf. \cite{Pi}. In \cite{Cae_w2} we can find  definition of $s$-numbers extended to the context of quasi-Banach spaces.

Now we recall the well known properties of the above $s$-numbers.
\begin{proposition}
 Let $s_k \in\{a_k,c_k, d_k, x_k\}$, $W, X, Y$ be  quasi-Banach spaces and let $Z$ be a $t$-Banach space,  $t\in(0,1]$. Then

\begin{enumerate}[1.]

\item $ \|T\|=s_1(T)\geq s_2(T)\geq s_3(T)\geq\ldots\geq 0$\quad
 for all $T~\in~L(X,Y)$,
	\item {\normalfont{(additivity)}} $$s_{n+k-1}^t (T_1 +T_2)\leq s_k^t(T_1) + s_n^t(T_2)\quad  $$
	  for all  $T_1,T_2~\in~L(X,Z)$ and $n,k\in\N$,
	\item {\normalfont(multiplicativity)} $$s_{n+k-1} (T_1T_2)\leq s_k(T_1)s_n(T_2)\quad $$
	 for all $T_1\in L(X,Y)$, $T_2\in L(W,X)$ and $n,k\in\N$,

	\item  $x_k(T) \leq c_k(T)\leq a_k(T),\quad d_k(T)\leq a_k(T)$\quad
	 for all $T~\in~L(X,Y)$ and $k\in\N$.

\end{enumerate}

\label{prop}

\end{proposition}

Estimates of  the Weyl numbers of embeddings of finite dimensional $\ell^N_p$, $N\in\N$,  will be very useful for us therefore we recall the already known estimates. Lemma \ref{xn1} and Lemma \ref{xn8} can be found in Caetano \cite{Cae_w3}. Lemma
\ref{xn2} is taken from K$\ddot{\rm o}$nig \cite[p. 186]{Ko86}, see
also in \cite{Cae_w3}.

\begin{lem}\label{xn1}
 Let $0< p_1\leq \max(2,p_2)\leq\infty$ and $1\leq k\leq N/2$, $N\in\N$. Then
\begin{equation}
x_k(\id:\ell_{p_1}^N\rightarrow \ell_{p_2}^N)\sim\left\{ \begin{array}{ll} k^{\frac{1}{p_2}-\frac{1}{p_1}}, & {\mathrm{if}\quad 0< p_1\leq p_2 \leq 2,} \\ 1,& {\mathrm{if}\quad 2\leq p_1\leq p_2\leq\infty,} \\ k^{\frac{1}{2}-\frac{1}{p_1}},&{\mathrm{if}\quad 0<p_1\leq  2\leq p_2\leq\infty,} \\ N^{\frac{1}{p_2}-\frac{1}{p_1}},& {\mathrm{if}\quad 0<	 p_2\leq p_1 \leq2.}\end{array}\right.
\label{eq:ali3}
\end{equation}
\end{lem}

\begin{lem}\label{xn2}
Let $2\le p_2< p_1\le\infty$. Then there is a positive constant $C$ independent of
$N$ and $k$ such that for all $k,N\in\N$
\begin{equation}\label{upper21}
 x_k\left({\rm id}: \ell_{p_1}^N \rightarrow  \ell_{p_2}^N\right)
\le C \left( N/k\right)^{(1/{p_2}-1/{p_1})/(1-2/{p_1})}.
\end{equation}
\end{lem}

\begin{lem}
\label{xn8}
Let $0 < p_2 \leq 2 < p_1 \leq\infty$. Then there is a positive constant $C$ independent of
$N$ and $k$ such that for all $k,N\in\N$
\begin{equation}
x_k(\id:\ell_{p_1}^N\rightarrow \ell_{p_2}^N)\geq C \left\{ \begin{array}{ll} N^{\frac{1}{p_2} - \frac{1}{2}}, & {\mathrm{if}\quad  k\leq N/2,} \\ N^{\frac{1}{p_2} - \frac{1}{p_1}},&{\mathrm{if}\quad k\leq N^{\frac{2}{p_1}},} \end{array}\right.
\label{eq:ali5}
\end{equation}
\end{lem}

\begin{lem}\label{xn4}
Let $2< p_2< p_1\le\infty$. Then there is a positive constant $C$ independent of
$N$ and $k$ such that for all $k,N\in\N$
\begin{equation} \label{low2p2p1}
 x_k\left({\rm id}: \ell_{p_1}^N \rightarrow  \ell_{p_2}^N\right)
\ge C  \ \ \ {\rm if}\ \
k\le N/4.
\end{equation}
\end{lem}
\begin{proof}
Using the multiplicativity of the Weyl numbers, together with Lemma
\ref{xn1} and the hypothesis $k\le\frac{N}{4}$, we
can write
\begin{equation*}
\begin{split}
C_1 \le x_{2k}\left({\rm id}: \ell_2^N \rightarrow  \ell_{p_2}^N\right)&\le
x_k\left({\rm id}: \ell_{p_1}^N \rightarrow  \ell_{p_2}^N\right)x_k\left({\rm
id}: \ell_2^N \rightarrow  \ell_{p_1}^N\right)\\ &\le C_2 x_k\left({\rm id}: \ell_{p_1}^N \rightarrow  \ell_{p_2}^N\right),
\end{split}
\end{equation*}
which completes the proof.
\end{proof}

The following result was already proved in the case $1\le p_2\le 2<
p_1\le\infty$; see \cite[p. 186]{Ko86}.

\begin{lem}\label{xn5}
Let $0< p_2\le 2< p_1\le\infty$. Then there is a positive constant $C$ independent of
$N$ and $k$ such that for all $k,N\in\N$
\begin{equation}
 x_k\left({\rm id}: \ell_{p_1}^N \rightarrow  \ell_{p_2}^N\right)
\le C N^{\frac{1}{{p_2}}}k^{-\frac{1}{2}}.
\end{equation}
\end{lem}
\begin{proof}
Using the multiplicativity of the Weyl numbers, together with Lemma~\ref{xn1}, Lemma \ref{xn2} and the hypothesis $k\le\frac{N}{2}$, we
can write
\begin{equation*}
\begin{split}
x_{2k}\left({\rm id}: \ell_{p_1}^N \rightarrow  \ell_{p_2}^N\right) &\le
x_k\left({\rm id}: \ell_2^N \rightarrow  \ell_{p_2}^N\right)x_k\left({\rm id}:
\ell_{p_1}^N \rightarrow  \ell_2^N\right)\\
&\le C N^{\frac{1}{{p_2}}}k^{-\frac{1}{2}}\left( N/k\right)^{(1/2-1/{p_1})/(1-2/{p_1})}
\le C N^{\frac{1}{{p_2}}}k^{-\frac{1}{2}},\end{split}
\end{equation*}

which finishes the proof, by the monotonicity of the Weyl numbers.
\end{proof}

The following lemma for Gelfand numbers will be useful for the upper estimates of Weyl numbers.

\begin{lem}\label{gn}
If~ $1\le k\le N<\infty$\ and $0<p_2\le p_1\le\infty,$ then
$$
c_k\big({\rm id}: \ell_{p_1}^N \rightarrow  \ell_{p_2}^N\big)=
(N-k+1)^{\frac{1}{p_2}-\frac{1}{p_1}}.
$$
\end{lem}
The proof of this lemma follows literally \cite[Section
11.11.4]{Pie78}, see also \cite{Pin85}.
Indeed the original proof is used only to deal with the Banach setting. However, the same proof works also in the quasi-Banach setting  $0<p_2\le p_1\le\infty$.

Based on the definitions of Weyl and Gelfand numbers, it is known that for any linear continuous operator $T$ between two complex quasi-Banach spaces, $x_k(T)\le c_k(T)$. Then in terms of Lemma \ref{xn2}, Lemma \ref{xn5} and Lemma \ref{gn}, we have the following two propositions.

\begin{proposition}\label{xn6}
Let $0< p_2\le 2< p_1\le\infty$. Then there is a positive constant $C$ independent of
$N$ and $k$ such that for all $k,N\in\N$
\begin{equation}\label{eq:ali6}
 x_k\left({\rm id}: \ell_{p_1}^N \rightarrow  \ell_{p_2}^N\right)
\le C
\begin{cases}
N^{\frac 1{p_2}-\frac 1{p_1}},
 & {\rm if}\, ~ 1\le k \le N^{\frac 2{p_1}},\\
N^{\frac 1{p_2}}k^{-\frac 12},
 & {\rm if}\, ~ N^{\frac 2{p_1}}\le k\le N.
\end{cases}
\end{equation}
\end{proposition}

\begin{proposition}\label{xn7}
Let $2\leq p_2< p_1\le\infty$. Then there is a positive constant $C$ independent of
$N$ and $k$ such that for all $k,N\in\N$
\begin{equation}\label{wn2p2p1}
 x_k\left({\rm id}: \ell_{p_1}^N \rightarrow  \ell_{p_2}^N\right)
\le C
\begin{cases}
N^{\frac 1{p_2}-\frac 1{p_1}},
 & {\rm if}\, ~ 1\le k \le N^{\frac 2{p_1}},\\
\left( \frac Nk\right)^{(1/{p_2}-1/{p_1})/(1-2/{p_1})},
 & {\rm if}\, ~ N^{\frac 2{p_1}}\le k\le N.
\end{cases}
\end{equation}
\end{proposition}

Of course when $k > N$ then
$$x_k(\id:\ell_{p_1}^N\rightarrow \ell_{p_2}^N)=0.$$

  \section{The embeddings of sequence spaces}

  In this part we consider the embeddings of sequence spaces  from Definition~\ref{ciagi2}, that will play an important role in this paper. In \cite {HL} we find the following theorem.

\begin{thm}
Let $0<p_1, p_2\leq \infty$, $0<q_1,q_2 \leq\infty$. $\{M_j\}_{j=0}^\infty$ be an arbitrary sequence  of natural numbers and $\{\beta_j\}_{j=0}^\infty$ be an arbitrary weight sequence. The embedding
$$\id: \ell_{q_1}(\beta_j\ell_{p_1}^{M_j})\rightarrow \ell_{q_2}(\ell_{p_2}^{M_j})$$
is compact if and only if
$$\big\{\beta_j^{-1}M_j^{({1}/{p_2}-{1}/{p_1})_+}\big\}_{j=0}^\infty\in \ell_{q^*}\quad \mathrm{if}\quad q^*<\infty$$ or
$$\lim_{j\rightarrow\infty}(\beta_j^{-1}M_j^{({1}/{p_2}-{1}/{p_1})_+})=0\quad \mathrm{if}\quad q^*=\infty,$$
where
$$\frac{1}{q^*}=\bigg(\frac{1}{q_2}-\frac{1}{q_1}\bigg)_+\quad\mathrm{i.e.}\quad q^*=\left\{ \begin{array}{ll} \infty, & {\mathrm{if}\quad 0< q_1\leq q_2 \leq\infty,} \\(q_1q_2)/(q_1-q_2),& {\mathrm{if}\quad 0< q_2<q_1<\infty,} \\ q_2,&{\mathrm{if}\quad q_1=\infty.} \end{array}\right.$$
\end{thm}

We need also the following lemma which will be used in the proof of the next theorem.
\begin{lem}
Let $\big\{\beta_j^{-1}M_j^{({1}/{p_2}-{1}/{p_1})_+}\big\}_{j=0}^\infty\in \ell_{q^*}$ and
$$\id_j x= (\delta _{j,k}x_{k,l})_{k\in\N_0,l=1,\ldots,M_k}=\left\{ \begin{array}{ll} 0, & { \mathrm{if}\ k\neq j,}\\ x_{j,l}, & { \mathrm{if}\ k= j\ \mathrm{and}\  l=1,\ldots,M_j.}\end{array}\right.$$
Then
$$\biggl|\bigg|(\id-\sum_{j=0}^N\id_j)x|\ell_{q_2}(\ell_{p_2}^{M_j})\biggr|\bigg|\leq \big|\big|\big\{\beta_j^{-1}M_j^{({1}/{p_2}-{1}/{p_1})_+}\big\}_{j=N+1}^\infty|\ell_{q^\ast}\big|\big| \cdot||x|\ell_{q_1}(\beta_j\ell_{p_1}^{M_j})||.$$
\label{lem:1}
\end{lem}
\noindent
The proof can be found in \cite{HL}.

Now we consider the asymptotic behavior of the  Weyl, the Gelfand numbers and the approximation numbers for these embeddings. We start with the estimates of the Weyl numbers.
\begin{thm}
\label{q0.2}
Let  $0< p_1, p_2\leq\infty, 0<q_1, q_2\leq\infty$ and $0<b<\infty$, $\delta>b(\frac{1}{p_2}-\frac{1}{p_1})_+$. Suppose $\lambda={(\frac 1{p_2}-\frac  1{p_1})/(1-\frac 2{p_1})}$,  $M_j\sim 2^{jb}$ and that the operator $\id: \ell_{q_1}(2^{j\delta}\ell_{p_1}^{M_j})\rightarrow \ell_{q_2}(\ell_{p_2}^{M_j})$ is compact.
\begin{enumerate}[(i)]
\item If $0< p_1\leq \max(2,p_2)\leq\infty$  or $0< p_2\le 2< p_1\le\infty$, then
$$
x_k(\id: \ell_{q_1}(2^{j\delta}\ell_{p_1}^{M_j})\rightarrow \ell_{q_2}(\ell_{p_2}^{M_j}))\sim k^{-\beta},$$
where
$$
\beta= \left\{ \begin{array}{ll} \frac{\delta}{b}+\frac{1}{p_1} - \frac{1}{p_2}, & {\text{if}\quad 0< p_1,p_2 \leq 2,} \\ \frac{\delta}{b},& {\text{if}\quad 2\leq p_1\leq p_2\leq\infty,} \\ \frac{\delta}{b}+\frac{1}{p_1}-\frac{1}{2},&{\text{if}\quad 0< p_1\leq 2\leq p_2\leq\infty,} \\ \frac{p_1}{2}(\frac{\delta}{b}+\frac{1}{p_1}-\frac{1}{p_2}), &{\text{if}\quad 0< p_2\le 2< p_1\le\infty\ \text{and}\ \delta<\frac{b}{p_2},}\\
 \frac{\delta}{b}+\frac{1}{2}-\frac{1}{p_2}, &{\text{if}\quad 0< p_2\le 2< p_1\le\infty\ \text{and}\  \delta>\frac{b}{p_2}.}\end{array}\right.
$$
\item If $2< p_2<p_1\le\infty$ and $\delta> b\lambda$, then
$$x_k({\rm id}:\ \ell_{q_1}(2^{j\delta}\ell ^{M_j}_{p_1})\rightarrow \ell_{q_2}(\ell ^{M_j}_{p_2}))\sim k^{-\frac{\delta}b}.$$

\item If $2< p_2<p_1\le\infty$ and $\delta< b\lambda$, then
$$c k^{-\frac{\delta}b}\le x_k({\rm id}:\ \ell_{q_1}(2^{j\delta}\ell ^{M_j}_{p_1})\rightarrow \ell_{q_2}(\ell ^{M_j}_{p_2}))\le C k^{-\frac{p_1}2(\frac{\delta}b+\frac 1{p_1}-\frac 1{p_2})}.$$
\end{enumerate}
\end{thm}

\begin{proof}

{\em Step 1.} \textit{Estimation from below.} We consider the commutative diagram
\begin{displaymath}
\xymatrix{
\ell_{p_1}^{M_j}\ar[r]^{S^j}\ar[d]_{\id^{(j)}} &
\ell_{q_1}(2^{j\delta}\ell_{p_1}^{M_j}) \ar[d]^{\id} \\
\ell_{p_2}^{M_j}& \ell_{q_2}(\ell_{p_2}^{M_j}), \ar[l]_{T_j}}
\end{displaymath}
 where $\id^{(j)}$ denotes the embedding from $\ell_{p_1}^{M_j}$ in $\ell_{p_2}^{M_j}$.
The operator  $T_j $ is a projection, such that
\begin{displaymath}
T_j x=\{\delta_{j,k} x_{k,l}\}_{l=1}^{M_j},
 \end{displaymath}
 and  $S^j$ is the natural embedding which maps $\{x_l\}_{l=1}^{M_j}$ to the $j$-th block in  $\ell_{q_1}(2^{j\delta}\ell_{p_1}^{M_j})$,
 \begin{displaymath}
 S^j(\{x_l\}_{l=1}^{M_j})= (\hat{x}_{v,l}),\ \text{where}
 \end{displaymath}
 \begin{displaymath}
\hat{x}_{v,l}= \left\{ \begin{array}{ll} 0, & {\text{if}\ v\neq j,} \\ x_l,& {\text{if}\ v=j\ \text{and}\ 1\leq l\leq M_j.}  \end{array}\right.
\end{displaymath}
Then
\begin{displaymath}
||S^j:\ell_{p_1}^{M_j}\rightarrow \ell_{q_1}(2^{j\delta}\ell_{p_1}^{M_j}) ||=2^{j\delta},\quad || T_j: \ell_{q_2}(\ell_{p_2}^{M_j}) \rightarrow \ell_{p_2}^{M_j}||=1
\end{displaymath}
and
\begin{displaymath}
\id^{(j)}=T_j\circ \id\circ S^{j}.
\end{displaymath}
Consequently
\begin{equation*}
x_k(\id^{(j)}:\ell_{p_1}^{M_j}\rightarrow \ell_{p_2}^{M_j})\leq 2^{j\delta} x_k(\id:\ell_{q_1}(2^{j\delta}\ell_{p_1}^{M_j})\rightarrow \ell_{q_2}(\ell_{p_2}^{M_j})).
\end{equation*}

Let $0< p_1\leq \max(2,p_2)\leq\infty$.
By (\ref{eq:ali3}) with $k=[M_j/2]$, we obtain
$$
x_k(\id)\geq C \left\{ \begin{array}{ll} k^{-({\delta}/{b}+{1}/{p_1} - {1}/{p_2})}, & {\text{if}\quad 0<p_1,p_2 \leq 2,} \\ k^{-{\delta}/{b}},& {\text{if}\quad 2\leq p_1\leq p_2\leq\infty,} \\ k^{-({\delta}/{b}+{1}/{p_1}-{1}/{2})},&{\text{if}\quad 0< p_1\leq 2\leq p_2\leq\infty.}  \end{array}\right.
$$

Let $0< p_2\le 2< p_1\le\infty.$\\
If $\delta<\frac{b}{p_2}$ then by  (\ref{eq:ali5})  with $k=[M_j^{{2}/{p_1}}]$, we obtain
$$
C k^{-{(p_1/2)}({\delta}/{b}+{1}/{p_1}-{1}/{p_2})}\leq C 2^{-j\delta}2^{jb({1}/{p_2}-{1}/{p_1})}\leq C 2^{-j\delta}M_j^{{1}/{p_2}-{1}/{p_1}}\leq x_{k}(\id).
$$
If $\delta>\frac{b}{p_2}$ then  the formula (\ref{eq:ali5}) with $k=[M_j/2]$, yields
$$
C k^{-({\delta}/{b}+{1}/{2}-{1}/{p_2})}\leq C 2^{-j\delta}M_j^{{1}/{p_2}-{1}/{2}}\leq x_{k}(\id).
$$

Let $2< p_2<p_1\le\infty$.
By (\ref{low2p2p1}) with $k=[M_j/4]$, we obtain

$$C k^{-\delta/b}\le C 2^{-j\delta}\le x_k({\rm id}).$$

{\em Step 2.} \textit{Estimation from above.} Let $\id_j$ be defined as in Lemma \ref{lem:1}. Then we have
\begin{equation*}
\biggl|\bigg|(\id-\sum_{j=0}^N\id_j)x|\ell_{q_2}(\ell_{p_2}^{M_j})\biggr|\bigg|\leq E_N\left|\left|x|\ell_{q_1}(2^{j\delta}\ell_{p_1}^{M_j})\right|\right|,
\end{equation*}
where
\begin{equation*}
E_N=\bigg|\bigg|\{2^{-j\delta}M_j^{({1 }/{p_2}-{1}/{p_1})_+}\}_{j=N+1}^{\infty}|\ell_{q*}\bigg|\bigg|.
\end{equation*}
Let $\rho=\min(1,p_2,q_2)$, then $l_{q_2}(l_{p_2}^{M_j})$ is a $\rho$-Banach space. Therefore from the properties of the Weyl numbers, we get
\begin{equation}
x_k^\rho(\id:\ell_{q_1}(2^{j\delta}\ell_{p_1}^{M_j})\rightarrow \ell_{q_2}(\ell_{p_2}^{M_j}))\leq E_N^\rho + \sum_{j=0}^Lx_{k_j}^\rho(\id_j)+ \sum_{j=L+1}^Nx_{k_j}^\rho(\id_j),
\label{q5}
\end{equation}
where
\begin{equation}
k=\sum_{j=0}^Nk_j-(N+1).
\label{*}
\end{equation}
We will choose $N$ later.\\
{\em Substep 2.1.} We consider the commutative diagram
\begin{displaymath}
\xymatrix{
\ell_{q_1}(2^{j\delta}\ell_{p_1}^{M_j}) \ar[r]^{\quad T_j}\ar[d]_{\id_j} &
\ell_{p_1}^{M_j} \ar[d]^{\id^{(j)}} \\
\ell_{q_2}(\ell_{p_2}^{M_j}) & \ell_{p_2}^{M_j},\ar[l]_{\quad S^j}}
\end{displaymath}
where $T_j$ and $S_j$ are  defined  similarly to $T_j$ and $S_j$  in Step 1. Now
\begin{displaymath}
||T_j:\ell_{q_1}(2^{j\delta}\ell_{p_1}^{M_j})\rightarrow \ell_{p_j}^{M_j}||=2^{-j\delta}\quad\text{and}\quad  ||S^j:\ell_{p_2}^{M_j}\rightarrow \ell_{q_2}(\ell_{p_2}^{M_j})||=1.
\end{displaymath}
Therefore $\id_j=S^j\circ \id^{(j)}\circ T_j$ and
\begin{equation}
x_{k_j}(\id_j)\leq 2^{-j\delta}x_{k_j}(\id^{(j)}:\ell_{p_1}^{M_j}\rightarrow \ell_{p_2}^{M_j}).
\label{q6}
\end{equation}
Consequently we get from  (\ref{q5})
\begin{multline*}
x_k^\rho(\id:\ell_{q_1}(2^{j\delta}\ell_{p_1}^{M_j})\rightarrow \ell_{q_2}(\ell_{p_2}^{M_j}))\\ \leq E_N^\rho + \sum_{j=0}^L2^{-j\delta\rho} x_{k_j}^\rho(\id^{(j)})+ \sum_{j=L+1}^N2^{-j\delta\rho}x_{k_j}^\rho(\id^{(j)}),
\end{multline*}
with arbitrary $L, N$ and $k_j$ satisfying (\ref{*}).\\
{\em Substep 2.2.} Now we estimate the sum $\sum_{j=0}^L x_{k_j}^\rho(\id_j)$.

Let $0< p_1\leq \max(2,p_2)\leq\infty$. First we show that there exists a constant $C>0$ independent of $N$ and $k$, such that
\begin{equation}
x_k(\id:\ell_{p_1}^N\rightarrow \ell_{p_2}^N)\leq C \left\{ \begin{array}{ll} k^{{1}/{p_2}-{1}/{p_1}}, & {\text{if}\quad 0\leq p_1\leq p_2 \leq 2,} \\ 1,& {\text{if}\quad 2\leq p_1\leq p_2\leq\infty,} \\ k^{{1}/{2}-{1}/{p_1}},&{\text{if}\quad 0< p_1\leq 2\leq p_2\leq\infty,} \\ N^{{1}/{p_2}-{1}/{p_1}},& {\text{if}\quad 0<	 p_2\leq p_1 \leq 2,}
\end{array}\right.
\label{eq:qali5}
\end{equation}
for $k\leq N$.
We consider the commutative diagram
\begin{displaymath}
\xymatrix{
\ell_{p_1}^N \ar[r]^{S}\ar[d]_{\id} &
\ell_{p_1}^{2N} \ar[d]^{\Id} \\
\ell_{p_2}^N & \ell_{p_2}^{2N},\ar[l]_{T}}
\end{displaymath}
where
$$S(\lambda_1,\ldots,\lambda_N)=(\lambda_1,\ldots,\lambda_N,0,\ldots,0)\quad\text{and}$$
$$T(\lambda_1,\ldots,\lambda_{2N})=(\lambda_1,\ldots,\lambda_N).$$
Both norms $||S||$ and $||T||$ are equal to 1. Therefore $x_k(\id)\leq x_k(\Id)$, and then by (\ref{eq:ali3}) we get (\ref{eq:qali5}).

Let
\begin{equation*}
k_j=[M_j2^{(L-j)\epsilon}]\quad\text{for}\quad j=0,1,\ldots,L,
\end{equation*}
with  $0<\epsilon<b$. Then
\begin{equation}
\sum_{j=0}^L k_j\leq \sum_{j=0}^L [c2^{jb}2^{(L-j)\epsilon}]\leq  [c2^{Lb}] \sum_{j=0}^L 2^{(j-L)(b-\epsilon)}\leq [c2^{Lb}].
\label{27c}
\end{equation}
We put
\begin{equation}
t=\left\{ \begin{array}{ll} {1}/{p_2} - {1}/{p_1}, & {\text{if}\quad 0< p_1, p_2 \leq 2,} \\ 0,& {\text{if}\quad 2\leq p_1\leq p_2\leq\infty,} \\ {1}/{2}-{1}/{p_1},&{\text{if}\quad 0< p_1\leq 2\leq p_2\leq\infty.} \end{array}\right.
\label{t}
\end{equation}
 So by  (\ref{q6}) and (\ref{eq:qali5}) we get
\begin{equation}
\label{11}
x_{k_j}(\id_j)\leq C 2^{-j\delta}M_j^t2^{(L-j)\epsilon t},
\end{equation}
if $k_j\leq M_j$ and
\begin{equation*}
x_{k_j}(\id_j)= 0,
\end{equation*}
if $k_j> M_j$.
Moreover if  $j<L-\frac{1}{\epsilon}$, then $k_j=[M_j2^{(L-j)\epsilon}]>M_j$ and then $x_{k_j}(\id_j)= 0$. Therefore  by (\ref{11}) we have: \\
for $t< 0$
\begin{multline}
\sum_{j=0}^L x_{k_j}^\rho(\id_j)=\sum_{j=\left\lceil L-{1}/{\epsilon}\right\rceil}^{L} x_{k_j}^\rho(\id_j)\\ \leq c (2^{-L\delta}2^{Lbt})^\rho \sum_{j=\left\lceil L-{1}/{\epsilon}\right\rceil}^{L}(2^{(L-j)\delta}2^{-(L-j)bt}2^{(L-j)\epsilon t})^\rho\\ \leq c 2^{-Lb\rho({\delta}/{b}-t)}\sum_{j=\left\lceil L-{1}/{\epsilon}\right\rceil}^{L}(2^{{\delta}/{\epsilon}}2^{-{bt}/{\epsilon}}2^{(L-j)\epsilon t})^\rho
\leq C 2^{-Lb\rho({\delta}/{b}-t)};
\label{q12}
\end{multline}
for $t=0$
\begin{multline}
\sum_{j=0}^L x_{k_j}^\rho(\id_j)=\sum_{j=\left\lceil L-{1}/{\epsilon}\right\rceil}^{L} x_{k_j}^\rho(\id_j)\leq c 2^{-L\rho\delta} \sum_{j=\left\lceil L-{1}/{\epsilon}\right\rceil}^{L}2^{(L-j)\rho\delta} \\ \leq c 2^{-L\rho\delta}\sum_{j=\left\lceil L-{1}/{\epsilon}\right\rceil}^{L}2^{\rho{\delta}/{\epsilon}} \leq c 2^{-L\rho\delta}\frac{1}{\epsilon}2^{\rho{\delta}/{\epsilon}}
\leq C 2^{-L\rho\delta};
\label{q12c}
\end{multline}
for $t>0$
\begin{multline}
\hspace{-0,2cm}\sum_{j=0}^L x_{k_j}^\rho(\id_j)=\sum_{j=\left\lceil L-{1}/{\epsilon}\right\rceil}^{L} x_{k_j}^\rho(\id_j)\leq c (2^{-L\delta}2^{Lbt})^\rho \sum_{j=\left\lceil L-{1}/{\epsilon}\right\rceil}^{L}(2^{(L-j)\delta}2^{-(L-j)bt})^\rho\\ \leq c 2^{-Lb\rho({\delta}/{b}-t)}\sum_{j=\left\lceil L-{1}/{\epsilon}\right\rceil}^{L}(2^{{\delta}/{\epsilon}}2^{-(L-j)bt})^\rho
\leq C 2^{-Lb\rho({\delta}/{b}-t)};
\label{q12b}
\end{multline}
where $\left\lceil x\right\rceil=\inf\{k\in\N: k\geq x\}$ and constant $C$ is dependent on $\epsilon, \delta, t$ but not on $L$.

Let $0< p_2\le 2< p_1\le\infty$.\\
If $\frac{\delta}{b}<\frac{1}{p_2}$ then we take
\begin{equation*}
k_j= \big[M_j^{{2}/{p_1}}2^{(L-j)\epsilon}\big]\quad\text{for}\quad j=0,1,\ldots,L,
\end{equation*}
with $\epsilon>0$. Since $\frac{\delta}{b}<\frac{1}{p_2}$ we can choose $\epsilon>0$, such that $\frac{\delta}{b}+\frac{1}{p_1}-\frac{1}{p_2}<\frac{\epsilon}{2b}<\frac{1}{p_1}$. So we have
\begin{equation}
\sum_{j=0}^L k_j\leq \sum_{j=0}^L [c 2^{jb({2}/{p_1})}2^{(L-j)\epsilon}]\leq  [c 2^{Lb({2}/{p_1})}] \sum_{j=0}^L 2^{b(j-L)({2}/{p_1}-{\epsilon}/{b})}\leq [c 2^{Lb({2}/{p_1})}].
\label{q27}
\end{equation}
Now by (\ref{q6}) and (\ref{eq:ali6})  we get
\begin{multline}
\label{q28}
x_{k_j}(\id_j)\leq C 2^{-j\delta}M_j^{{1}/{p_2}-{1}/{p_1}} 2^{-(L-j)(\epsilon/2)} \\ \leq C 2^{-L\delta}2^{Lb({1}/{p_2}-{1}/{p_1})}2^{(L-j)\delta}2^{(L-j)b({1}/{p_1}-{1}/{p_2})}2^{-(L-j)(\epsilon/2)},
\end{multline}
since $M_j^{{2}/{p_1}}\leq k_j$. Hence
\begin{multline}
\label{q29}
\sum_{j=0}^L x_{k_j}^\rho(\id_j)\leq c 2^{-Lb\rho({\delta}/{b} +{1}/{p_1}-{1}/{p_2})} \sum_{j=0}^L2^{b\rho(L-j)({\delta}/{b}+{1}/{p_1}-{1}/{p_2}-{\epsilon}/{(2b)})}\\ \leq C 2^{-Lb\rho({\delta}/{b} +{1}/{p_1}-{1}/{p_2})},
\end{multline}
with constant $C$ independent of $L$.\\
If $\frac{\delta}{b}>\frac{1}{p_2}$ then we take
\begin{equation}
k_j= [M_j2^{(L-j)\epsilon}]\quad\text{for}\quad j=0,1,\ldots,L,
\label{q32}
\end{equation}
with a fixed $\epsilon$, $0<\epsilon<b$. Hence
\begin{equation}
\sum_{j=0}^L k_j\leq \sum_{j=0}^L [c 2^{jb}2^{(L-j)\epsilon}]\leq  [c 2^{Lb}] \sum_{j=0}^L 2^{(j-L)(b-\epsilon)}\leq [c 2^{Lb}].
\label{q27b}
\end{equation}
If $j<L- \frac{1}{\epsilon}$ then $k_j>M_j$. Therefore from (\ref{q6}) and (\ref{eq:ali6})  we get
\begin{multline*}
\sum_{j=0}^Lx_{k_j}^\rho(\id_j)=\sum_{j=\left\lceil L-{1}/{\epsilon}\right\rceil}^{L}x_{k_j}^\rho(\id_j)\\ \leq c (2^{-L\delta}2^{Lb({1}/{p_2}-{1}/{2})})^\rho\sum_{j=\left\lceil L-{1}/{\epsilon}\right\rceil}^{L}(2^{(L-j)\delta}2^{-(L-j)b({1}/{p_2}-{1}/{2})}2^{-(L-j)(\epsilon/2)})^\rho.
\end{multline*}
Here
\begin{multline*}
\sum_{j=\left\lceil L-{1}/{\epsilon}\right\rceil}^{L}(2^{(L-j)\delta}2^{-(L-j)b({1}/{p_2}-{1}/{2})}2^{-(L-j)(\epsilon/2)})^\rho\\
\leq 2^{\rho{\delta}/{\epsilon}}\sum_{j=\left\lceil L-{1}/{\epsilon}\right\rceil}^{L}(2^{-(L-j)b({1}/{p_2}-{1}/{2})}2^{-(L-j)(\epsilon/2)})^\rho\leq C.
\end{multline*}
So \begin{equation}
\sum_{j=0}^Lx_{k_j}^\rho(\id_j)\leq C 2^{-Lb\rho(\delta/b + {1}/{2}-{1}/{p_2})}.
\label{qc6}
\end{equation}
Let $2< p_2<p_1\le\infty$.\\
If $\delta< b\lambda$ then we take
$$k_j=[M_j^{2/p_1}2^{(L-j)\varepsilon}]\ \ {\rm for}\ \ j=0, 1,\ldots, L,$$
with $\varepsilon>0$. Since $\delta< b\lambda$ we can choose $\varepsilon>0$, such that $\frac 1\lambda(\frac\delta b+\frac 1{p_1}-\frac 1{p_2})<\frac{\varepsilon}b<\frac 2{p_1}$. We should mention that the equality $\frac 1\lambda(\frac\delta b+\frac 1{p_1}-\frac 1{p_2})=\frac 2{p_1}$ holds true if and only if $\delta= b\lambda$. Then
\begin{equation}\label{kLle}\sum^L_{j=0}k_j\le [c2^{Lb(2/p_1)}].\end{equation}
By means of (\ref{wn2p2p1}) and (\ref{q6}), we have

$$
x_{k_j}({\rm id}_j)\le C 2^{-j\delta}M_j^{\lambda(1-2/p_1)}2^{-(L-j)\varepsilon\lambda} \leq C 2^{-Lb(\frac{\delta}b+\frac1{p_1}-\frac1{p_2})}2^{b(L-j)(\frac{\delta}b+\frac 1{p_1}-\frac 1{p_2}-\frac{\varepsilon\lambda}b)}.
$$

Therefore,
\begin{equation}\label{xkLle}\sum_{j=0}^{L}x_{k_j}^\rho({\rm id}_j)\le C 2^{-Lb\rho (\frac{\delta}b+\frac 1{p_1}-\frac 1{p_2})},
\end{equation}
with constant $C$ independent of $L$.\\
If $\delta> b\lambda$ then we take
$$k_j=[M_j2^{(L-j)\varepsilon}]\ \ {\rm for}\ \ j=0, 1,\ldots, L,$$
with fixed $\varepsilon,\ 0<\varepsilon<b$. Then $\delta> \lambda\varepsilon$ and
\begin{equation}\label{kLge}\sum^L_{j=0}k_j\le [c2^{Lb}].\end{equation}
If $j<L-\frac 1\varepsilon$, then $k_j>M_j$. In view of (\ref{upper21}) and (\ref{q6}), we have
\begin{multline}\label{xkLge}
\sum^L_{j=0}x_{k_j}^\rho({\rm id}_j) =\sum^L_{j=\lceil L-1/\varepsilon\rceil}x_{k_j}^\rho({\rm id}_j)
 \le C_1\sum^L_{j=\lceil L-1/\varepsilon\rceil}(2^{-j\delta}2^{-(L-j)\lambda\varepsilon})^\rho
\\
 \le C_1 2^{-L\delta\rho}\sum^L_{j=\lceil L-1/\varepsilon\rceil}(2^{(L-j)(\delta-\lambda\varepsilon)})^\rho
\le C_2 2^{-L\delta\rho}2^{(\delta-\lambda\varepsilon)\rho/\varepsilon}\le C_3 2^{-L\delta\rho}.
\end{multline}

{\em Substep 2.3.} Now we estimate the sum $\sum_{j=L+1}^N x_{k_j}^\rho(\id_j)$ and choose $N$ such that $E_N$ is small enough.

Consider the first case $0< p_1\leq \max(2,p_2)\leq\infty$.
We show that
\begin{equation}
\sum_{j=L+1}^N x_{k_j}^\rho(\id_j)\leq C 2^{-Lb\rho({\delta}/{b}-t)},
\label{q15}
\end{equation}
where $t$ is defined in (\ref{t}).
Let
$$k_j=\max\big\{\big[M_L(j-L)^{-2}\big], 1\big\}\quad\text{for}\quad j=L+1,\ldots,N.$$
Then
\begin{equation}
\sum_{j=L+1}^N  k_j\leq c M_L + (N-L).
\label{q22}
\end{equation}
If  $t\leq 0$ then $k_j^t\leq (M_L(j-L)^{-2})^t$. Therefore from   (\ref{q6}) and (\ref{eq:qali5}) we get
\begin{equation}
\sum_{j=L+1}^N x_{k_j}^\rho(\id_j)\leq c( 2^{-L\delta}2^{Lbt})^\rho\sum_{j=L+1}^N (2^{(L-j)\delta}(j-L)^{-2t})^\rho\leq C 2^{-Lb\rho({\delta}/{b}-t)}.
\label{qc3}
\end{equation}
If $t>0$ and  since $\delta>b(\frac{1}{p_2}-\frac{1}{p_1})_+=\frac{b}{t}$ then from  (\ref{q6}) and (\ref{eq:qali5})  we obtain
\begin{equation}
\sum_{j=L+1}^N x_{k_j}^\rho(\id_j)\leq c (2^{-L\delta}2^{Lbt})^\rho\sum_{j=L+1}^N(2^{(L-j)\delta}2^{-(L-j)bt})^\rho\leq C 2^{-Lb\rho({\delta}/{b}-t)},
\label{qc4}
\end{equation}
where $C$ is independent of $L$ and $N$.

 Now we can choose $N$ such that  $E_N\leq c2^{-L\delta}M_L^t$. This is possible because
$\{E_N\}_{N=1}^\infty$ is a decreasing sequence such that $\lim_{N\rightarrow\infty}E_N=0.$
The formulas (\ref{*}), (\ref{27c}) and (\ref{q22}) imply that $k\leq[cM_L]$.  So by (\ref{q5}), (\ref{q12}), (\ref{q12c}), (\ref{q12b}), (\ref{qc3}) and (\ref{qc4}) we have
$x_{[c M_L]}(\id)\leq C 2^{-Lb({\delta}/{b}-t)}.$ Finally
$$x_{k}(\id)\leq C k^{-({\delta}/{b}-t)}, \quad \text{for all}\ k\in\N.$$

Consider the next case $0< p_2\le 2< p_1\le\infty$.\\
If $\frac{\delta}{b}<\frac{1}{p_2}$ then we take
$$k_j=\max\big\{\big[M_L^{{2}/{p_1}}(j-L)^{-2}\big],1\big\} \quad\text{for}\quad j=L+1,\ldots,N.$$
Hence
\begin{equation}
\sum_{j=L+1}^N  k_j\leq c M_L^{{2}/{p_1}} + (N-L).
\label{**}
\end{equation}
 Now by (\ref{eq:ali6}), (\ref{q6}) and since $\delta>b(\frac{1}{p_2}-\frac{1}{p_1})_+$ we have
\begin{multline}
\sum_{j=L+1}^N x_{k_j}^\rho(\id_j)\leq C  (2^{-L\delta}2^{Lb({1}/{p_2}-{1}/{p_1})})^\rho\sum_{j=L+1}^N(2^{(L-j)\delta}2^{(L-j)b({1}/{p_1}-{1}/{p_2})})^\rho\\  \leq c 2^{-Lb\rho({\delta}/{b}+{1}/{p_1}-{1}/{p_2})}\sum_{j=L+1}^N2^{b\rho(L-j)({\delta}/{b}+{1}/{p_1}-{1}/{p_2})}\leq c 2^{-Lb\rho({\delta}/{b}+{1}/{p_1}-{1}/{p_2})}.
\hspace{-0,2cm}\label{qc5}
\end{multline}
As in the previous case we can choose $N$ such that $E_N\leq c 2^{-L\delta}{M_L}^{{1}/{p_2}-{1}/{p_1}}.$
The formulas (\ref{*}), (\ref{q27}) and (\ref{**}) imply  $k\leq [c M_L^{2/p_1}]$. So by (\ref{q5}), (\ref{q29}) and (\ref{qc5}) we have
$x_{[c M_L^{2/p_1}]}(\id)\leq C 2^{-Lb({\delta}/{b}+{1}/{p_1}-{1}/{p_2})}.$
Finally
$$x_{k}(\id)\leq C k^{-({p_1}/{2})({\delta}/{b}+{1}/{p_1}-{1}/{p_2})}, \quad \text{for all}\ k\in\N.$$
But if we assume  $\frac{\delta}{b}>\frac{1}{p_2}$, then taking
 $$k_j=\max\big\{[M_L(j-L)^{-2}],1\big\}\quad\text{for}\quad j=L+1,\ldots,N,$$
we get
\begin{equation}
\sum_{j=L+1}^N  k_j\leq c M_L + (N-L).
\label{***}
\end{equation}
 We notice that if $1\leq k_j\leq M_j^{{2}/{p_1}}$ then $M_j^{{1}/{p_2}-{1}/{p_1}}\leq M_j^{{1}/{p_2}}k_j^{-{1}/{2}}$. So by   (\ref{eq:ali6}) and (\ref{q6}) we have
\begin{multline}
\sum_{j=L+1}^N x_{k_j}^\rho(\id_j)\leq \sum_{j=L+1}^N c ( 2^{-j\delta}2^{(bj)/{p_2}}2^{-(bL)/{2}}(j-L))^\rho\\ \leq c 2^{-Lb\rho({\delta}/{b}-{1}/{p_2}+{1}/{2})}\sum_{j=L+1}^N (2^{b(L-j)({\delta}/{b}-{1}/{p_2})}(j-L))^\rho \\ \leq C 2^{-Lb\rho({\delta}/{b}-{1}/{p_2}+{1}/{2})}.
\label{qc7}
\end{multline}
Now we choose $N$ such that $E_N\leq c 2^{-L\delta}{M_L}^{{1}/{p_2}-{1}/{2}}$.
From (\ref{*}), (\ref{q27b}), (\ref{***}) we have $k\leq [cM_L]$.
So by (\ref{q5}), (\ref{qc6}), (\ref{qc7}) we get
$x_{[cM_L]}(\id)\leq C 2^{-Lb({\delta}/{b}+{1}/{2}-{1}/{p_2})}.$  Then
$$x_{k}(\id)\leq C k^{-({\delta}/{b}+{1}/{2}-{1}/{p_2})}, \quad \text{for all}\ k\in\N.$$

Consider the last case  $2< p_2<p_1\le\infty$.\\
If $\delta< b\lambda$ then we take
$$k_j=\max\{[M_L^{2/p_1}(j-L)^{-2}],1\}\ \ {\rm for}\ \ j=L+1,\ldots, N.$$
Hence, $k_j\le M_L^{2/p_1}\le M_j^{2/p_1}$ and
\begin{equation}\label{kNle}\sum^N_{j=L+1}k_j\le c M_L^{2/p_1}+(N-L).\end{equation}
Based on (\ref{wn2p2p1}), (\ref{q6}) and the previous assumption $\delta>b(\frac 1{p_2}-\frac 1{p_1})_+$,  we have
\begin{equation}\label{xkNle}\sum^N_{j=L+1}x_{k_j}^\rho({\rm id}_j)\le C \sum^N_{j=L+1}(2^{-j\delta}2^{jb(\frac 1{p_2}-\frac 1{p_1})})^\rho\le C 2^{-Lb\rho(\frac{\delta}b+\frac 1{p_1}-\frac 1{p_2})}.
\end{equation}
We choose $N$ such that $E_N\le c2^{-L\delta}M_L^{\frac 1{p_2}-\frac 1{p_1}}$. The formulas (\ref{kLle}) and (\ref{kNle}) imply that $k=\sum_{j=0}^Nk_j-(N+1)\le [cM_L^{2/p_1}]$. Finally, by
(\ref{q5}), (\ref{xkLle}) and (\ref{xkNle}), we have
$$x_k({\rm id})\le Ck^{-\frac{p_1}2(\frac{\delta}b+\frac 1{p_1}-\frac 1{p_2})}, \ \ \ {\rm for\ all}\ k\in\mathbb{N}.$$
If $\delta> b\lambda$ then we take
$$k_j=\max\{[M_L(j-L)^{-2}],1\}\ \ {\rm for}\ \ j=L+1,\ldots, N,$$
and we get
\begin{equation}\label{kNge}
\sum^N_{j=L+1}k_j\le cM_L+(N+L).
\end{equation}
By (\ref{upper21}) and (\ref{q6}), we have
\begin{equation}\label{xkNge}
\begin{split}
\sum^N_{j=L+1}x_{k_j}^\rho({\rm id}_j)
&\le C \sum^N_{j=L+1}(2^{-j\delta}2^{jb\lambda}2^{-Lb\lambda}(j-L)^{2\lambda})^\rho\\
&\le C 2^{-L\rho\delta}\sum^N_{j=L+1}2^{\rho(L-j)(\delta-b\lambda)}(j-L)^{2\rho\lambda}
\le C 2^{-L\rho\delta}.
\end{split}
\end{equation}
We choose $N$ such that $E_N\le c2^{-L\delta}$. The formulas (\ref{kLge}) and (\ref{kNge}) yield that $k=\sum_{j=0}^Nk_j-(N+1)\le [cM_L]$. Then, by (\ref{q5}),
(\ref{xkLge}) and (\ref{xkNge}), we have
$$x_k({\rm id})\le Ck^{-\frac{\delta}b}, \ \ \ {\rm for\ all}\ k\in\mathbb{N}.$$

\end{proof}

\begin{rmk} Still, there are minor gaps left open in the estimates for the case $\delta<b\lambda$. It should be mentioned that, for the bounds of the asymptotic order, the inequality $\frac\delta b=\frac{p_1}2(\frac\delta b+\frac 1{p_1}-\frac 1{p_2})$ holds if and only if $\delta= b\lambda$.
Furthermore, the problem becomes much more complicated if $\delta=b\lambda$. How do the Weyl numbers behave in such limiting situation? Different
from those previous cases, the corresponding behaviour herein maybe depend on the parameters
$q_1$ and $q_2$ to a certain extent.
\end{rmk}

The asymptotic behavior of the approximation, the Gelfand and the Kolmogorov numbers can be found in \cite[Theorem 3.5, 4.12 and 4.6]{JV}. Since a bit different notation is used there we recall here the results using our symbols. Some partial results were proved earlier in \cite{Cae}.
\begin{thm}
\label{q0.3}
Let\ ~$0< p_1,p_2\leq\infty,\ 0<q_1, q_2\leq\infty$ and $0<b<\infty$, $\delta>b(\frac{1}{p_2}-\frac{1}{p_1})_+$. Suppose $M_j\sim 2^{jb}$,
$\frac 1p=\frac{1}{\min(p_1',p_2)}$ and  that the operator $\id:~\ell_{q_1}(2^{j\delta}\ell_{p_1}^{M_j})\rightarrow \ell_{q_2}(\ell_{p_2}^{M_j})$ is compact.
Then
$$
a_k(\id: \ell_{q_1}(2^{j\delta}\ell_{p_1}^{M_j})\rightarrow \ell_{q_2}(\ell_{p_2}^{M_j}))\sim k^{-\beta},
$$
where
\begin{equation}
\beta= \left\{ \begin{array}{ll}  \frac{\delta}{b}, &{\text{if}\quad 0< p_1\leq p_2\leq 2\ \text{or}\ 2\leq p_1\leq p_2\leq\infty,}\\ \frac{p}{2}\frac{\delta}{b}, &{\text{if}\quad 0< p_1<2< p_2< \infty\  \text{and}\ \delta<\frac{b}{p}\ \text{or}}\\ & {\hspace{+0,6cm}  1< p_1<2< p_2= \infty\  \text{and}\ \delta<\frac{b}{p},}\\
 \frac{\delta}{b}+\frac{1}{2}-\frac{1}{p}, &{\text{if}\quad 0< p_1<2< p_2\leq \infty\ \text{and}\ \delta>\frac{b}{p},}\\ \frac{\delta}{b}+\frac{1}{p_1}-\frac{1}{p_2}, &{\text{if}\quad 0< p_2\leq p_1\leq \infty,}\\
 \frac{\delta}{b}+\frac{1}{2}-\frac{1}{p_2}, &{\text{if}\quad 0< p_1\leq 1< p_2= \infty.}
 \end{array}\right.
 \label{qc9}
\end{equation}
\end{thm}

\begin{thm}
\label{q0.4}
Let $0< p_1,p_2\leq\infty, 0<q_1, q_2\leq\infty$ and $0<b<\infty$, $\delta>b(\frac{1}{p_2}-\frac{1}{p_1})_+$. Suppose $M_j\sim 2^{jb}$, $\theta=\big({\frac{1}{p_1}-\frac{1}{p_2}}\big)/\big({\frac{1}{p_1}-\frac{1}{2}}\big)$ and that the operator $\id: \ell_{q_1}(2^{j\delta}\ell_{p_1}^{M_j})\rightarrow \ell_{q_2}(\ell_{p_2}^{M_j})$ is compact.
Then

$$c_k(\id:  \ell_{q_1}(2^{j\delta}\ell_{p_1}^{M_j})\rightarrow \ell_{q_2}(\ell_{p_2}^{M_j}))\sim k^{-\beta},
$$
where
\begin{equation*}
\beta=\left\{ \begin{array}{ll} \frac{\delta}{b}, &{\text{if}\quad 2\leq p_1<p_2\leq\infty,}\\ \frac{\delta}{b} + \frac{1}{p_1}- \frac{1}{p_2}, &  {\text{if}\quad 0< p_2\leq p_1\leq\infty\   \text{or}} \\  &  {\hspace{+0,6cm} 0<p_1<p_2\leq2\ \text{and}\ \frac{\delta}{b}>\frac{\theta}{p_1'}, } \\
\frac{p_1'}{2}\frac{\delta}{b},&  {\text{if}\quad  1<p_1<p_2\leq2\ \text{and}\  \frac{\delta}{b}<\frac{\theta}{p_1'}\ \text{or}}  \\ & {\hspace{+0,6cm} 1< p_1< 2 < p_2 \leq\infty\ \text{and}\   \frac{\delta}{b}<\frac{1}{p_1'}, }\\
 \frac{\delta}{b} +\frac{1}{p_1}- \frac{1}{2}, & {\text{if}\quad 0< p_1< 2 < p_2 \leq\infty\ \text{and}\  \frac{\delta}{b}>\frac{1}{p_1'}. } \end{array}\right.
\end{equation*}
\end{thm}

\begin{thm}
\label{q0.5}
Let $0< p_1,p_2\leq\infty, 0<q_1, q_2\leq\infty$ and $0<b<\infty$, $\delta>b(\frac{1}{p_2}-\frac{1}{p_1})_+$. Suppose $M_j\sim 2^{jb}$, $\theta'=\big({\frac{1}{p_1}-\frac{1}{p_2}}\big)/\big({\frac{1}{2}-\frac{1}{p_2}}\big)$ and that the operator $\id: \ell_{q_1}(2^{j\delta}\ell_{p_1}^{M_j})\rightarrow \ell_{q_2}(\ell_{p_2}^{M_j})$ is compact.
Then

$$d_k(\id:  \ell_{q_1}(2^{j\delta}\ell_{p_1}^{M_j})\rightarrow \ell_{q_2}(\ell_{p_2}^{M_j}))\sim k^{-\beta},
$$
where
\begin{equation*}
\beta=\left\{ \begin{array}{ll}  \frac{\delta}{b}, &{\text{if}\quad  0< p_1\leq p_2\leq 2,}\\ \frac{\delta}{b} +  \frac{1}{2}-\frac{1}{p_2}, &{\text{if}\quad 0< p_1<2< p_2\leq\infty\ \text{and}\ \frac{\delta}{b}>\frac{1}{p_2},}\\  \frac{p_2}{2}\frac{\delta}{b}, &{\text{if}\quad 0< p_1<2< p_2<\infty\ \text{and}\ \frac{\delta}{b}<\frac{1}{p_2}\ \text{or}} \\ &\hspace{+0,6cm}{2< p_1\leq p_2\leq \infty\  \text{and}\ \frac{\delta}{b}<\frac{\theta' }{p_2},}\\ \frac{\delta}{b}+\frac{1}{p_1}-\frac{1}{p_2}, &{\text{if}\quad 2< p_1\leq p_2\leq \infty\  \text{and}\ \frac{\delta}{b}>\frac{\theta' }{p_2}\ \text{or}}\\ &\hspace{+0,6cm} {0< p_2\leq p_1 \leq\infty.} \end{array}\right.
\end{equation*}
\end{thm}

\section{The embeddings of function spaces on quasi-bounded domains}
In this part, we use the estimates from the previous section to obtain the asymptotic behavior of the $s$-numbers of the following embeddings
\begin{equation*}
\bar{A}^{s_1}_{p_1,q_1}(\Omega)\hookrightarrow\bar{A}^{s_2}_{p_2,q_2}(\Omega),
\end{equation*}
defined on the quasi-bounded domains. Theorem 3.23 in \cite{HT} guarantees  that
$$
s_k(\bar{A}^{s_1}_{p_1,q_1}(\Omega)\hookrightarrow\bar{A}^{s_2}_{p_2,q_2}(\Omega))\sim s_k(\ell_{q_1}(2^{j\delta}\ell_{p_1}^{M_j})\hookrightarrow \ell_{q_2}(\ell_{p_2}^{M_j})),
$$
for any  $s$-numbers, if $\Omega$ is the uniformly $E$-porous domain  described in  Definition \ref{Epor}. If domain $\Omega$ is quasi-bounded and the assumptions of  Theorem \ref{twr:1.4} are satisfied then the embedding is compact. Moreover if
\begin{equation}
0<\liminf_{j\rightarrow\infty}b_j(\Omega)2^{-jb(\Omega)}\leq \limsup_{j\rightarrow\infty}b_j(\Omega)2^{-jb(\Omega)}<\infty,
\label{b_j}
\end{equation}
then using (\ref{l_kostek}) we get
\begin{equation*}
M_j\sim 2^{jb(\Omega)}.
\end{equation*}

Using all the above remarks  and the well known elementary embeddings
$$
B_{p,q_1}^s(\R^d)\hookrightarrow F_{p,q}^s(\R^d)\hookrightarrow B_{p,q_2}^s(\R^d),
\quad\text{if}\ q_1\leq \min (p,q),\  q_2\geq \max(p,q),
$$
we get the following theorem.

\begin{thm}
\label{main}
Let $\Omega$ be the quasi-bounded domain satisfying (\ref{b_j}) and uniformly  $E$-porous  in $\R^d$ with $\Omega\neq\R^d$. Let $0< p_1,p_2\leq\infty, 0< q_1, q_2\leq\infty$, $b=b(\Omega)<\infty$, $\delta=s_1-s_2-d(\frac{1}{p_1}-\frac{1}{p_2})>\frac{b(\Omega)}{p^*}$, $\frac{1}{p^*}=\big(\frac{1}{p_2}-\frac{1}{p_1}\big)_{+}$, $\lambda={(\frac 1{p_2}-\frac  1{p_1})/(1-\frac 2{p_1})}$,  $\theta=\big({\frac{1}{p_1}-\frac{1}{p_2}}\big)/\big({\frac{1}{p_1}-\frac{1}{2}}\big)$, $\theta'=\big({\frac{1}{p_1}-\frac{1}{p_2}}\big)/\big({\frac{1}{2}-\frac{1}{p_2}}\big)$ and $\frac{1}{p}=\frac{1}{\min(p_1^{'},p_2)}$.
Then
\begin{enumerate}[(i)]
	\item
$$
 x_k(\bar{A}^{s_1}_{p_1,q_1}(\Omega)\hookrightarrow\bar{A}^{s_2}_{p_2,q_2}(\Omega))\sim k^{-\gamma},
$$

where
$$\hspace{-0.8cm}\gamma=\left\{ \begin{array}{ll}\frac{s_1-s_2}{b}+\frac{b-d}{b}(\frac{1}{p_1}-\frac{1}{p_2})
, & {\text{if}\quad 0< p_1, p_2 \leq 2,} \\ \frac{s_1-s_2}{b}-\frac{d}{b}(\frac{1}{p_1}-\frac{1}{p_2}),& {\text{if}\quad 2\leq p_1\leq p_2\leq\infty\ \text{or}}\\ & {\hspace{+0,57cm}  2< p_2<p_1\le\infty\  \text{and}\ \delta>b\lambda,} \\ \frac{s_1-s_2}{b}+\frac{b(\frac{1}{p_1}-\frac{1}{2})-d(\frac{1}{p_1}-\frac{1}{p_2})}{b},& {\text{if}\quad 0< p_1\leq 2\leq p_2\leq\infty,}\\
\frac{p_1}{2}\big(\frac{s_1-s_2}{b}+\frac{b-d}{b}(\frac{1}{p_1}-\frac{1}{p_2})\big), & {\text{if}\quad 0< p_2\le 2< p_1\le\infty\ \text{and}\  \frac{\delta}{b}<\frac{1}{p_2},}\\
 \frac{s_1-s_2}{b}+\frac{b(\frac{1}{2}-\frac{1}{p_2})-d(\frac{1}{p_1}-\frac{1}{p_2})}{b}, & {\text{if}\quad 0< p_2\le 2< p_1\le\infty\ \text{and}\ \frac{\delta}{b}>\frac{1}{p_2}.}\end{array}\right. $$
 \item If $2< p_2<p_1\le\infty$ and $\delta< b\lambda$, then
 \begin{multline*}
 c k^{-\frac{s_1-s_2}{b}+\frac{d}{b}(\frac{1}{p_1}-\frac{1}{p_2})}\leq x_k(\bar{A}^{s_1}_{p_1,q_1}(\Omega)\hookrightarrow\bar{A}^{s_2}_{p_2,q_2}(\Omega))\\ \leq C k^{-\frac{p_1}{2}\big(\frac{s_1-s_2}{b}+\frac{b-d}{b}(\frac{1}{p_1}-\frac{1}{p_2})\big) }.
\end{multline*}
 \item
 $$
 a_k(\bar{A}^{s_1}_{p_1,q_1}(\Omega)\hookrightarrow\bar{A}^{s_2}_{p_2,q_2}(\Omega))\sim k^{-\gamma},
$$
 where
$$\hspace{-0.6cm}\gamma=  \left\{ \begin{array}{ll} \frac{s_1-s_2}{b}-\frac{d}{b}(\frac{1}{p_1}-\frac{1}{p_2}), & {\text{if}\quad 0< p_1\leq p_2\leq 2\ \text{or}\ 2\leq p_1\leq p_2\leq\infty,} \\ \frac{p}{2}\big(\frac{s_1-s_2}{b}-\frac{d}{b}(\frac{1}{p_1}-\frac{1}{p_2})\big),&  {\text{if}\quad 0< p_1<2< p_2< \infty\ \text{and}\  \frac{\delta}{b}<\frac{1}{p}\ \text{or} }\\  & \hspace{+0.55cm}  {1< p_1<2< p_2= \infty\ \text{and}\  \frac{\delta}{b}<\frac{1}{p},}\\
\frac{s_1-s_2}{b}+\frac{b(\frac{1}{2}-\frac{1}{p})-d(\frac{1}{p_1}-\frac{1}{p_2})}{b},&  {\text{if}\quad  0< p_1<2< p_2\leq \infty\ \text{and}\  \frac{\delta}{b}>\frac{1}{p}, }\\  \frac{s_1-s_2}{b}+\frac{b-d}{b}(\frac{1}{p_1}-\frac{1}{p_2}),& {\text{if}\quad 0< p_2\leq p_1\leq \infty,}\\
  \frac{s_1-s_2}{b}+\frac{b(\frac{1}{2}-\frac{1}{p_2})-d(\frac{1}{p_1}-\frac{1}{p_2})}{b},&  {\text{if}\quad  0< p_1\leq 1< p_2= \infty.}\end{array}\right.
$$

 \item
 $$
 c_k(\bar{A}^{s_1}_{p_1,q_1}(\Omega)\hookrightarrow\bar{A}^{s_2}_{p_2,q_2}(\Omega))\sim k^{-\gamma},
$$
 where
$$\hspace{-0.2cm}\gamma=  \left\{ \begin{array}{ll} \frac{s_1-s_2}{b}-\frac{d}{b}(\frac{1}{p_1}-\frac{1}{p_2}), & {\text{if}\quad 2\leq p_1<p_2\leq\infty,}\\ \frac{s_1-s_2}{b}+\frac{b-d}{b}(\frac{1}{p_1}-\frac{1}{p_2}), & {\text{if}\quad 0< p_2\leq p_1 \leq\infty\  \text{or}} \\ & \hspace{+0.55cm}  {0<p_1<p_2\leq2\ \text{and}\  \frac{\delta}{b}>\frac{\theta}{p_1'}, } \\
\frac{p_1'}{2}\big(\frac{s_1-s_2}{b}-\frac{d}{b}(\frac{1}{p_1}-\frac{1}{p_2})\big),&  {\text{if}\quad  1<p_1<p_2\leq2\ \text{and}\  \frac{\delta}{b}<\frac{\theta}{p_1'}\ \text{or} }\\ & \hspace{+0,55cm} {1< p_1< 2 < p_2 \leq\infty\ \text{and}\ \frac{\delta}{b}<\frac{1}{p_1'},}\\
\frac{s_1-s_2}{b}+\frac{b(\frac{1}{p_1}-\frac{1}{2})-d(\frac{1}{p_1}-\frac{1}{p_2})}{b}, & {\text{if}\quad 0< p_1<2 < p_2 \leq\infty\ \text{and}\ \frac{\delta}{b}>\frac{1}{p_1'}. } \end{array}\right.
$$
 \item
  $$
 d_k(\bar{A}^{s_1}_{p_1,q_1}(\Omega)\hookrightarrow\bar{A}^{s_2}_{p_2,q_2}(\Omega))\sim k^{-\gamma},
$$
where
$$\hspace{-0.5cm}\gamma=\left\{ \begin{array}{ll}  \frac{s_1-s_2}{b}-\frac{d}{b}(\frac{1}{p_1}-\frac{1}{p_2}), &{\text{if}\quad  0< p_1\leq p_2\leq 2,}\\ \frac{s_1-s_2}{b} + \frac{b( \frac{1}{2}-\frac{1}{p_2})-d(\frac{1}{p_1}-\frac{1}{p_2})}{b}, &{\text{if}\quad 0< p_1<2< p_2\leq\infty\ \text{and}\ \frac{\delta}{b}>\frac{1}{p_2},}\\  \frac{p_2}{2}(\frac{s_1-s_2}{b}-\frac{d}{b}(\frac{1}{p_1}-\frac{1}{p_2})), &{\text{if}\quad 0< p_1<2< p_2<\infty\ \text{and}\ \frac{\delta}{b}<\frac{1}{p_2}\ \text{or}} \\ &\hspace{+0,6cm}{2< p_1\leq p_2\leq \infty\  \text{and}\ \frac{\delta}{b}<\frac{\theta' }{p_2},}\\ \frac{s_1-s_2}{b}+\frac{b-d}{b}(\frac{1}{p_1}-\frac{1}{p_2}), &{\text{if}\quad 2< p_1\leq p_2\leq \infty\  \text{and}\ \frac{\delta}{b}>\frac{\theta' }{p_2}\ \text{or}}\\ &\hspace{+0,6cm} {0< p_2\leq p_1 \leq\infty.} \end{array}\right.
$$
 \end{enumerate}
\label{twr5}
\end{thm}

The above formulas simplify if $\Omega$ is a domain of finite measure.
\begin{Corollary}
Let $\Omega$ be uniformly $E$-porous in $\R^d$ of finite Lebesgue measure and $\delta=s_1-s_2-d(\frac{1}{p_1}-\frac{1}{p_2})_+>0$. Let $\lambda$, $\theta$, $\theta'$ and $\frac{1}{p}$ denote the same as in Theorem \ref{main}. Then
\begin{enumerate}[(i)]
	\item
$$
 x_k(\bar{A}^{s_1}_{p_1,q_1}(\Omega)\hookrightarrow\bar{A}^{s_2}_{p_2,q_2}(\Omega))\sim k^{-\gamma},
$$

where
$$\gamma=\left\{ \begin{array}{ll}\frac{s_1-s_2}{d}
, & {\text{if}\quad 0< p_1, p_2 \leq 2,} \\ \frac{s_1-s_2}{d}-\frac{1}{p_1}+\frac{1}{p_2},& {\text{if}\quad 2\leq p_1\leq p_2\leq\infty\ \text{or}}\\ & {\hspace{+0,57cm}  2< p_2<p_1\le\infty\  \text{and}\ \delta>d\lambda,} \\ \frac{s_1-s_2}{d}-\frac{1}{2}+\frac{1}{p_2},&{\text{if}\quad 0<p_1\leq 2\leq p_2\leq\infty,}\\ \frac{p_1}{2}\frac{s_1-s_2}{d}, &{\text{if}\quad 0< p_2\le 2< p_1\le\infty\ \text{and}\  \frac{\delta}{d}<\frac{1}{p_2},}\\
 \frac{s_1-s_2}{d}+\frac{1}{2}-\frac{1}{p_1}, &{\text{if}\quad 0< p_2\le 2< p_1\le\infty\ \text{and}\ \frac{\delta}{d}>\frac{1}{p_2}.}\end{array}\right. $$
 \item
 If $2< p_2<p_1\le\infty$ and $\delta< d\lambda$, then
$$
 c k^{ -\frac{s_1-s_2}{d}+\frac{1}{p_1}-\frac{1}{p_2}}\leq x_k(\bar{A}^{s_1}_{p_1,q_1}(\Omega)\hookrightarrow\bar{A}^{s_2}_{p_2,q_2}(\Omega)) \leq C k^{- \frac{p_1}{2}\frac{s_1-s_2}{d} }.
$$
 \item
 $$
 a_k(\bar{A}^{s_1}_{p_1,q_1}(\Omega)\hookrightarrow\bar{A}^{s_2}_{p_2,q_2}(\Omega))\sim k^{-\gamma},
$$
 where
$$\hspace{-0.35cm}\gamma=  \left\{ \begin{array}{ll} \frac{s_1-s_2}{d}-\frac{1}{p_1}+\frac{1}{p_2}, & {\text{if}\quad 0< p_1\leq p_2\leq 2\ \text{or}\ 2\leq p_1\leq p_2\leq\infty,} \\ \frac{p}{2}(\frac{s_1-s_2}{d}-\frac{1}{p_1}+\frac{1}{p_2}),&  {\text{if}\quad 0< p_1<2< p_2< \infty\ \text{and}\  \frac{\delta}{d}<\frac{1}{p}\ \text{or} }\\ & \hspace{+0,6cm}{ 1< p_1<2< p_2= \infty\ \text{and}\  \frac{\delta}{d}<\frac{1}{p},} \\
\frac{s_1-s_2}{d}+\frac{1}{2}-\frac{1}{p}-\frac{1}{p_1}+\frac{1}{p_2},&  {\text{if}\quad  0<p_1<2< p_2\leq \infty\ \text{and}\  \frac{\delta}{d}>\frac{1}{p}, }\\  \frac{s_1-s_2}{d},& {\text{if}\quad 0< p_2\leq p_1\leq \infty,}\\ \frac{s_1-s_2}{d}+\frac{1}{2}-\frac{1}{p_1},&  {\text{if}\quad  0<p_1\leq 1< p_2= \infty.} \end{array}\right.
$$

 \item
 $$
 c_k(\bar{A}^{s_1}_{p_1,q_1}(\Omega)\hookrightarrow\bar{A}^{s_2}_{p_2,q_2}(\Omega))\sim k^{-\gamma},
$$
where
$$\gamma=  \left\{ \begin{array}{ll} \frac{s_1-s_2}{d}-\frac{1}{p_1}+\frac{1}{p_2}, & {\text{if}\quad 2\leq p_1<p_2\leq\infty,}\\ \frac{s_1-s_2}{d}, & {\text{if}\quad 0< p_2\leq p_1 \leq\infty\  \text{or}} \\ & \hspace{+0,6cm}  {0<p_1<p_2\leq2\ \text{and}\  \frac{\delta}{d}>\frac{\theta}{p_1'}, } \\
\frac{p_1'}{2}(\frac{s_1-s_2}{d}-\frac{1}{p_1}+\frac{1}{p_2}),&  {\text{if}\quad  1<p_1<p_2\leq2\ \text{and}\  \frac{\delta}{d}<\frac{\theta}{p_1'}\ \text{or} }\\ & \hspace{+0,6cm} {1< p_1< 2 < p_2 \leq\infty\ \text{and}\ \frac{\delta}{d}<\frac{1}{p_1'},}\\
\frac{s_1-s_2}{d}-\frac{1}{2}+\frac{1}{p_2}, & {\text{if}\quad 0< p_1< 2 < p_2 \leq\infty\ \text{and}\ \frac{\delta}{d}>\frac{1}{p_1'}. } \end{array}\right.
$$
 \item
  $$
 d_k(\bar{A}^{s_1}_{p_1,q_1}(\Omega)\hookrightarrow\bar{A}^{s_2}_{p_2,q_2}(\Omega))\sim k^{-\gamma},
$$
where
$$\gamma=\left\{ \begin{array}{ll}  \frac{s_1-s_2}{d}-\frac{1}{p_1}+\frac{1}{p_2}, &{\text{if}\quad  0< p_1\leq p_2\leq 2,}\\ \frac{s_1-s_2}{d} +  \frac{1}{2}-\frac{1}{p_1}, &{\text{if}\quad 0< p_1<2< p_2\leq\infty\ \text{and}\ \frac{\delta}{d}>\frac{1}{p_2},}\\  \frac{p_2}{2}(\frac{s_1-s_2}{d}-\frac{1}{p_1}+\frac{1}{p_2}), &{\text{if}\quad 0< p_1<2< p_2<\infty\ \text{and}\ \frac{\delta}{d}<\frac{1}{p_2}\ \text{or}} \\ &\hspace{+0.6cm}{2<p_1\leq p_2\leq \infty\  \text{and}\ \frac{\delta}{d}<\frac{\theta' }{p_2},}\\ \frac{s_1-s_2}{d}, &{\text{if}\quad 2< p_1\leq p_2\leq \infty\  \text{and}\ \frac{\delta}{d}>\frac{\theta' }{p_2}\ \text{or}}\\ &\hspace{+0,6cm} {0< p_2\leq p_1 \leq\infty.} \end{array}\right.
$$
 \end{enumerate}

\end{Corollary}

\begin{rmk}
Estimates of the approximation, the Gelfand, the Kolmogorov and the Weyl numbers given in the last corollary coincide with the previous results for these numbers in the case of bounded Lipschitz domains; cf. \cite[Section 3.3.4]{ET}, \cite{Cae}, \cite[Theorem 3.5, 4.12 and 4.6]{JV}. To the best of our knowledge the estimates of the Weyl numbers were not complete to the very end; cf.  A.~Caetano \cite{Cae_w1,Cae_w2}. The above corollary complements  the previous estimates.
\end{rmk}

\section*{Acknowledgments}

The first author was partially supported by the Natural Science Foundation of China (Grant Nos
11171137, 61173187 and 61173188), the 211 Project of Anhui University (Grant No. 33050069)
and Doctoral Research Start-up Fund Project of Anhui University (Grant No. 33190215).

\bigskip

\bigskip

\noindent Shun Zhang

\noindent Key Laboratory of Intelligent Computing $\&$ Signal Processing of Ministry of Education, 
\noindent and School of Computer Science and Technology,

\noindent  Anhui University,

\noindent  Hefei, 230601 Anhui,

\noindent  P.R. CHINA

\noindent  shzhang27@163.com

\medskip

\noindent Alicja G\k{a}siorowska

\noindent Institute of Mathematics,

\noindent A. Mickiewicz University,

\noindent Umultowska 87, 61-614 Pozna\'n,

\noindent POLAND

\noindent  alig@amu.edu.pl
\end{document}